\documentclass[11pt,reqno]{amsart}
\usepackage[margin = 1.25in]{geometry}                
\usepackage{amsmath,amssymb}
\usepackage{epsfig}
\usepackage{graphics,color}
\usepackage{hyperref}
\usepackage{tikz}
\usetikzlibrary{arrows}

\input{epsf.sty}

\newcommand{\ds}{\displaystyle}
\newcommand{\inv}{\operatorname{inv}}
\renewcommand\P{\operatorname{\mathbb P{}}}
\newcommand{\B}[2]{\mathbb{B}^{#1,#2}}
\newcommand{\W}[2]{\mathbb{W}^{#1,#2}}
\newcommand{\al} {\alpha}
\newcommand{\be}{\begin{equation}}
\newcommand{\ee}{\end{equation}}

\newcommand{\refL}[1]{Lemma~\ref{#1}}
\newcommand{\refS}[1]{Section~\ref{#1}}

\newtheorem{theorem}{Theorem}[section]
\newtheorem{lemma}[theorem]{Lemma}
\newtheorem{prop}[theorem]{Proposition}
\newtheorem{corollary}[theorem]{Corollary}

\newtheorem{remark}[theorem]{Remark}

\newtheorem{open}[theorem]{Open problem}

\newtheorem{example}[theorem]{Example}

\begin{document}

\title{Reverse juggling processes}

\author{Arvind Ayyer} 
\address{Arvind Ayyer, Department of Mathematics, Indian Institute of Science, Bangalore - 560012, India}
\email{arvind@iisc.ac.in}

\author{Svante Linusson}
\address{Svante Linusson, Department of Mathematics, KTH-Royal Institute of Technology, 
  SE-100 44, Stockholm, Sweden.}
\email{linusson@math.kth.se}
\date{\today}

\begin{abstract}
Knutson introduced two families of reverse juggling Markov chains (single and multispecies) motivated by the study of random semi-infinite matrices over $\mathbb{F}_q$.
We present natural generalizations of both chains by placing generic weights that still lead to simple combinatorial expressions for the stationary distribution.
For permutations, this is a seemingly new multivariate generalization of the inversion polynomial.\end{abstract}
\maketitle

\section{Introduction}
\label{S:intro}
Juggling has been studied from different mathematical perspectives, e.g. from combinatorics \cite{Buhleretal, Stadler2}, probability \cite{Warrington, LeskelaVarpanen, EngstromLeskelaVarpanen, ABCLN, ABCN-2015}, and algebraic geometry\cite{KnutsonLamSpeyer}. 
In recent work, Knutson returns to the study of juggling inspired by a matrix model \cite{knutson-2018}.

In the matrix model defined by Knutson, we have a random
 semi-infinite (to the right) matrix with $b$ rows and entries from $\mathbb{F}
 \equiv GF(q)$ generated as follows. At each time step, a uniformly random column from $\mathbb{F}^b$ is chosen and added to the left of the current matrix. This is easily seen to be a Markov chain.
Knutson studies two projections of this chain. In the first,  the set of matrices is stratified by the columns where the rank increases, when going from left to right. The positions of these columns are denoted by the configuration ${\bf n}=(n_1,\dots, n_b)$, where $n_1 < \cdots < n_b$. As in Example \ref{E:matrix1}, a ball is positioned in every such column, which is placed above the matrix. When the matrix is extended with a new random column to the left, there will be shift of the balls. If the new column is in the linear span of the leftmost $n_i$ columns but not in the linear span of the leftmost $n_{i-1}$ columns then this will result in ball number $i$ moving to the front.
The Markov chain on matrices then projects to the one on increasing $b$-tuples
of integers with the following transitions rates.
\begin{equation}
\label{knutson-rates}
{\bf n} \to
\begin{cases}
(n_1+1,\dots,n_b+1) & \text{with probability $\frac{1}{q^b}$}, \\
(1, n_1+1,\dots,\widehat{n_i+1}, \dots, n_b+1) & \text{with probability $\frac{1}{q^{b-i}}-\frac{1}{q^{b-1+1}}$ for $1 \leq i \leq b$},
\end{cases}
\end{equation}
where $\hat x$ means that the element $x$ should be omitted. The first case happens if and only if the new column is the all zero column.
The movement of the balls is the time-reversed version of what has become known as a juggling Markov chain, called the Multivariate Juggling Markov Chain (MJMC) in \cite{ABCN-2015}.

\begin{example}
\label{E:matrix1}
Let $b=4, q=3$. In the example below the new column causes the third ball from the left to move to the front and all other balls move one step to the right. 
This happens with probability $1/3(1-1/3)$.
\[\renewcommand{\arraystretch}{1}
\begin{array} {c | c | c | c | c | c  | c | c} 
\bullet  & \bullet & &\bullet  & & & \bullet & \\
 \hline
\hline
1 & 2 & 0& 0& 0& 2& 0& \ldots \\
 0& 0 & 0& 1 & 1 & 1 & 0 & \ldots\\
 0&  1 & 2  & 0  & 2  & 0 & 1 & \ldots \\
 0& 0 & 0 & 0 & 0 & 0 & 2 & \ldots \\
 \hline
\end{array}
\longrightarrow
\begin{array} {c| c | c | c | c | c | c  | c | c} 
\color{red} \bullet  & \bullet & \bullet & & & & & \bullet & \\
 \hline
\hline
\color{red} 0&1 & 2 & 0& 0& 0& 2& 0& \ldots \\
\color{red} 1 &0& 0 & 0& 1 & 1 & 1 & 0 & \ldots\\
\color{red} 2 &0&  1 & 2  & 0  & 2  & 0 & 1 & \ldots \\
\color{red} 0 &0& 0 & 0 & 0 & 0 & 0 & 2 & \ldots \\
 \hline
\end{array}
\]
\end{example}

\medskip

Knutson showed that the stationary probability distribution of the reverse juggling chain is given by a simple formula. In Section~\ref{S:single-infinite}, we generalize this process by setting the jump probability of ball $j$ to $x_j$, which we call the Infinite Reverse Juggling Markov Chain (IRJMC). We show that this continues to be a nice solvable model, in the sense that the stationary distribution continues to have a simple expression. First though, we focus on the window of the the first $m$ positions ($m >b$) of the IRJMC in Section \ref{S:single-finite}, which we call the Reverse Juggling Markov Chain (RJMC) and prove a formula for the stationary distribution and a property of ultrafast mixing.

The second projection of the Markov chain on the semi-infinite matrices studied by Knutson, comes from a finer stratification of the space of matrices. One way to think of the first model is that we want to reduce the semi-infinite matrix to a matrix of zeros and $b$ $1$'s where we are allowed to use any row operation and rightward column operation. Here, by rightward column operation, we mean that we can add the content of column $i$ to column $j$, where $i<j$. The $1$'s will then be in the columns indicated by the balls  $(n_1,\dots, n_b)$. For the second projection, we allow only downward row operations and rightward column operations, and we then record the row, counted from above, where the 1 in that column is positioned. 
 We now think of the row number as the labelling of the ball. The Markov chain on matrices now projects onto a chain whose states are labelled balls. 
For this model, one may prove that the balls change by a bumping path as follows. 
A ball is chosen with the same probabilities as in the first chain as given in \eqref{knutson-rates} and moves to the left. As that ball moves to the left it will bump (replace) a ball with smaller label with probability 
$1-1/q$ and move on with probability $1/q$. Then the ball (or the bumped ball) will continue left and for the next ball with a smaller label, it will again either bump it or move on as above. 
This process continues until a ball reaches the front. 
See Example~\ref{E:matrix2} for an example of a state and a transition. 
For a proof that this gives the right transition probability, see \cite[Section 4]{knutson-2018}. Knutson generalised this process on labelled balls to one where there are potentially several balls carrying a particular label. He gave a simple product formula for the stationary distribution of this chain.

\begin{example}
\label{E:matrix2}
Let $b=4, q=3$. In the example below the new column causes the ball labelled 4 to start a bumping path including the balls labelled 3 and 1. This happens with probability $(1-1/3)^3 1/3$ because there are $2^3 = 8$ vectors which have a nonzero component along the vectors $(1,0,0,0)$, $(0,0,1,0)$ and $(0,0,0,1)$ and a zero component along $(0,1,0,0)$, out of a total of $3^4$ vectors.

\[\renewcommand{\arraystretch}{1.1}
\begin{array} {c | c | c | c | c | c} 
\raisebox{.5pt}{\textcircled{\raisebox{-1.0pt} {1}}}&
\raisebox{.5pt}{\textcircled{\raisebox{-1.0pt} {3}}}
& & \raisebox{.5pt}{\textcircled{\raisebox{-1.0pt} {2}}} & \raisebox{.5pt}{\textcircled{\raisebox{-1.0pt} {4}}} & \\
 \hline
\hline
1 & 0 & 1& 0&  0& \ldots \\
 0& 0 & 0& 1 &  0& \ldots\\
 0& 1 & 2& 0  &  0 &\ldots \\
 0& 0 & 0& 0  &  1 &\ldots \\
 \hline
\end{array}
\quad\longrightarrow\quad
\begin{array} {c| c | c | c | c | c| c} 
\color{red} 
\raisebox{.5pt}{\textcircled{\raisebox{-1.0pt} {1}}} &
\raisebox{.5pt}{\textcircled{\raisebox{-1.0pt} {3}}} &
\raisebox{.5pt}{\textcircled{\raisebox{-1.0pt} {4}}} & &
\raisebox{.5pt}{\textcircled{\raisebox{-1.0pt} {2}}} & &\\

 \hline
\hline
\color{red} 1 & 1& 0 & 1& 0& 0& \ldots \\
\color{red} 0 & 0& 0 & 0& 1 &  0&\ldots\\
\color{red} 1 & 0& 1 & 2& 0  &  0&\ldots \\
\color{red} 1 & 0& 0 & 0& 0  &  1 &\ldots \\
\hline
\end{array}
\]

\end{example}

\medskip
This process turns out to be closely related to the time-reversal of the Multispecies Juggling Markov Chain (MSJMC), which was studied in \cite{ABCLN} starting from a very different motivation. 
In Section \ref{S:multi-infinite}, we study the Infinite Multispecies Reverse Juggling Markov Chain (IMRJMC) which generalises Knutson's second model. This generalization is more intricate than that of the first model. We have two sets of variables for the transition probabilities, one for jumping and one for bumping. 
We prove explicit formulas for the stationary probability distribution, which turns out to have separate factors in these sets of variables.
A key step in the proof is the study of the stationary distribution of the same chain, where we ignore the empty spaces. We call this the Multispecies Reverse Juggling Markov Chain (MRJMC) and we study it in Section \ref{S:multi-finite}. 

Finally, we end with some remarks and suggest open problems in Section~\ref{S:open}.

\section{Reverse juggling Markov chain} 
\label{S:single-finite}

We first define the Reverse Juggling Markov Chain (RJMC).
Fix $m,b \in \mathbb{N}$ such that $b \leq m$ and an arbitrary probability distribution on $\{0,1,\dots,b\}$ with $\mathbb{P}(i) = x_i$. 
Let $\mathbb{B}^{m,b}$ be the set of binary words of length $m$ with at most $b$ ones. 
Let $w = (w_1,\dots,w_m)$ be a word with $\ell$ ones in the first $m-1$ positions.  The letter in the last position is irrelevant for the definition of the transitions, as long as the new word belongs to the state space. Then the transitions of the RJMC are as follows.

\begin{enumerate}
\item With probability $x_{0}$, go to state $(0,w_1,\dots,w_{m-1})$.

\item With probability $x_i$ for  $1 \leq i \leq \ell$, move the $i$'th one from the left to the front, replace it by a zero, and shift everything to the right.

\item With probability $x_{\ell+1}+\cdots +x_b$, go to state $(1,w_1,\dots,w_{m-1})$. (This clearly does not happen if $\ell=b$.)
\end{enumerate}

Let $z_i = x_0 + \cdots + x_i$ for $0 \leq i \leq b$ and $\bar z_i = 1- z_i =  x_{i+1} + \cdots + x_{b}$ for $0 \leq i \leq b$. 

\begin{prop}
\label{prop:RJMC-irred}
The RJMC is irreducible and aperiodic.
\end{prop}

\begin{proof}
It is clear that one can get from any binary word to $(0,\dots,0)$ by repeatedly
adding $0$ to the left. For the converse, one adds $0$'s and $1$'s to the left until one obtains the desired word. This proves the irreducibility. Since $(0,\dots,0)$ goes to itself with probability $x_0$, the chain is aperiodic.
\end{proof}

We denote the transition matrix and the stationary distribution of the RJMC by $M$ and $\pi$ respectively and recall that $\pi M = \pi$.

\begin{theorem}
\label{thm:RJMC-ss}
For a configuration $w$, let $k$ be the number of $1$'s in $w$, and the positions of the $1$'s be given by $n_1,\dots,n_k$.
Then the stationary distribution of the chain is given by
\be \label{RJMC-ss}
\pi(w) = \prod_{i=0}^{k-1} \bar z_{i}
\prod_{i=0}^{k} z_i^{n_{i+1}-n_{i}-1},
\ee
where $n_0 = 0$ and $n_{k+1} = m+1$.
\end{theorem}

\begin{example}
We illustrate Theorem~\ref{thm:RJMC-ss} by constructing the transition graph with $m=3$ and $b=2$ in Figure~\ref{fig:eg-32}.

\begin{figure}
\begin{center}
\includegraphics[scale=1]{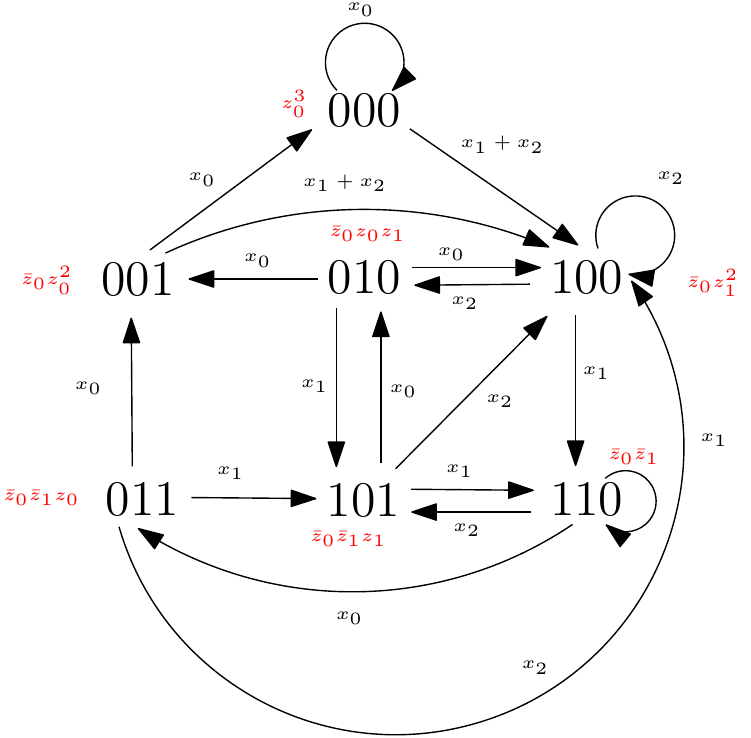}
\caption{The transition graph of the finite single species reverse juggling chain with $m=3$ and $b=2$. The transition probabilities are depicted next to the arrows. The stationary probabilities are shown in red.}
\label{fig:eg-32}
\end{center}
\end{figure}

\end{example}

The following result, which is proved by an easy computation,  will be useful in the proof of Theorem~\ref{thm:RJMC-ss}.

\begin{lemma}
\label{lem:sum-differ-at-last}
Let $\pi$ be defined as in \eqref{RJMC-ss}.
Let $v = (0,v_1,\dots,v_{m-1})$, with $k$ $1$'s in positions $n_1<\dots<n_{k}$.
Define 
 $v' = (v_1,\dots,v_{m-1},0)$ and $v'' = (v_1,\dots,v_{m-1},1)$. Then
\[
\pi(v') + \pi(v'') = \frac{1}{z_0} \pi(v) = \prod_{i=0}^{k-1} \left(\bar z_{i}
z_i^{s_{i+1}-s_{i}-1} \right)\; z_{k}^{m-s_{k}}.
\]
\end{lemma}

\begin{proof}[Proof of Theorem \ref{thm:RJMC-ss}]

By Proposition~\ref{prop:RJMC-irred}, the stationary distribution is unique. Hence, it suffices to verify that the probabilities given by \eqref{RJMC-ss} satisfy the master equation,
\be \label{mastereq}
\sum_{v \in \B mk} \P(v \to w) \pi(v) = \pi(w).
\ee

We will first suppose that the first $m-1$ positions in $w$ contain $k$ $1$'s, where $k$ is strictly less than $b$. There are two natural cases. First, suppose $w_1 = 0$. Then the only possibilities for $b$ are $(w_2,\dots,w_{n},0)$ and $(w_2,\dots,w_{n},1)$ and in both cases $\P(v \to w) = x_0$. Then the left hand side of \eqref{mastereq} is immediately equal to $\pi(w)$ using Lemma~\ref{lem:sum-differ-at-last}.
This completes the proof when 
$w_1 = 0$. 

We now consider the case when $w_1 = 1$, in which case, $n_1 = 1$. Then there are two types of transitions leading to $w$, where either a 1 is moved from the first $m-1$ sites to the front, or where a 1 is added to the left. Let us consider the former. There are two natural subcases depending on the location $\ell$ of the moving 1 in $v$.
\begin{itemize}

\item $1 \leq \ell \leq n_2-2$: For each such $\ell$, there are two possibilities for $v$ in \eqref{mastereq} depending on whether the last site is a 0 or 1. The remaining 1's in $b$ are at positions $l,n_2-1,\dots,n_k-1$. In both cases, they make a transition to $w$ with probability $x_1$ and adding these using Lemma~\ref{lem:sum-differ-at-last} gives us,
\[
x_1 \prod_{i=1}^{k} \bar z_{i-1} \; z_0^{\ell-1} z_1^{n_2-\ell-2} 
\prod_{i=2}^{k} z_i^{n_{i+1}-n_{i}-1},
\]
and summing this over the given range for $\ell$ gives
\[
\prod_{i=1}^{k} \bar z_{i-1} \; (z_1^{n_2-2} - z_0^{n_2-2}) 
\prod_{i=2}^{k} z_i^{n_{i+1}-n_{i}-1}.
\]

\item $2 \leq j \leq k$ and $n_j \leq \ell \leq n_{j+1}-2$: The transition to $w$ occurs here with probability $x_j$. As above, there are two possibilities for $v$ for each $\ell$ depending on the last site, and adding them using Lemma~\ref{lem:sum-differ-at-last}  gives
\[
x_j \prod_{i=1}^{k} \bar z_{i-1} \; z_0^{n_2-2} \prod_{i=1}^{j-2} z_i^{n_{i+2}-n_{i+1}-1}  \;
z_{j-1}^{\ell-n_j} z_{j}^{n_{j+1}-\ell-2} \;
\prod_{i=j+1}^{k} z_i^{n_{i+1}-n_{i}-1}.
\]
Summing these contributions over allowed positions of $\ell$ gives
\[
\prod_{i=1}^{k} \bar z_{i-1} \; 
 z_0^{n_2-2} \prod_{i=1}^{j-2} z_i^{n_{i+2}-n_{i+1}-1}  \;
(z_j^{n_{j+1}-n_j-1} - z_{j-1}^{n_{j+1}-n_j-1}) 
\prod_{i=j+1}^{k} z_i^{n_{i+1}-n_{i}-1}.
\]
Recall that $n_{k+1} = m+1$ and hence we have the contribution when $\ell$ is the rightmost 1 in the word.
\end{itemize}

Finally, summing these contributions over all possible values of $\ell$ telescopes leaving us with
\[
\prod_{i=1}^{k} \bar z_{i-1} \; \left(
 z_1^{n_2-2} \prod_{i=2}^{k} z_i^{n_{i+1}-n_{i}-1} -
  z_0^{n_2-2} \prod_{i=1}^{k-1} z_i^{n_{i+1}-n_{i}-1} 
 \right).
\]
The final case to consider is the one where a 1 is added to $v$, this time with probability 
$x_k + \cdots + x_b = \bar z_{k-1}$ since $v$ has $k-1$ 1's. 
Again, there are only two possibilities for $v$ depending on the last site. Adding these contributions using Lemma~\ref{lem:sum-differ-at-last} gives
\[
\bar z_{k-1} \prod_{i=1}^{k-1} \bar z_{i-1}  \; z_0^{n_2-2} \; \prod_{i=1}^{k-1} z_i^{n_{i+1}-n_{i}-1}.
\]
But adding this to the previous contribution and recalling that $n_1 = 1$ returns us exactly $\pi(w)$.

To complete the proof, we should consider the situation when $w$ has $b$ $1$'s in the first $m-1$ positions. As before, there are two subcases, depending on whether $w_1$ is $1$ or not. 
The only difference now is that Lemma~\ref{lem:sum-differ-at-last} is not applicable when $w_1 = 0$, because we cannot have more than $b$ $1$'s in $v$.
But the master equation is easily verified in this case.
The remaining calculations are similar in the other subcase to the situation when there are less than $b$ $1$'s in the first $m-1$ positions, and the proof goes through in the same way. 
\end{proof}

Theorem~\ref{thm:RJMC-ss} only shows that $\pi$ is a probability distribution up to normalisation. But it turns out that something stronger holds.

\begin{theorem}
\label{thm:RJMC-partfn}
$\pi$ is a probability distribution on $\B mb$. In other words,
\[
\sum_{a \in \B mb} \pi(a) = 
(x_0 + \cdots + x_{b})^m = 1.
\]
\end{theorem}

It turns out that the RJMC also satisfies a property called ultrafast convergence.

\begin{theorem}
\label{thm:RJMC-conv}
The RJMC on $\B mb$ converges to its stationary distribution in at most $m$ steps.
\end{theorem}

The proof of Theorems~\ref{thm:RJMC-partfn} and \ref{thm:RJMC-conv} will follow from the construction of an enriched Markov chain. The enriched chain then lumps (projects) down onto the original chain. For a formal definition of lumping see \cite[Lemma 2.5]{LPW}.
The strategy here follows closely that of \cite[Section~4.2]{ABCN-2015}. In particular, the case of $b=m$ coincides after ``particle-hole'' symmetry with the annihilation juggling model.
The next corollary follows because Theorem~\ref{thm:RJMC-conv} proves that $M^{m+1} = M^m$.

\begin{corollary}
\label{cor:RJMC-eigen}
All eigenvalues of $M$ are as follows: the eigenvalue 1 occurs with multiplicity one and all other eigenvalues are equal to 0.
\end{corollary}

\subsection{Enriched Chain on Words}
\label{S:enriched-single-finite}
Let $\W mb$ consist of words of length $m$ in the symbols $\{1,\dots,b+1\}$. Using the distribution on $\{0,\dots,b\}$ given by $\P(\cdot)$, we define a Markov chain on $\W mb$ by the following transitions. For $\tau = (\tau_1,\dots,\tau_m) \in \W mb$,
\[
\P(\tau \to (t,\tau_1,\dots,\tau_{m-1})) = x_{t-1} \quad \text{for $1 \leq t \leq b+1$}.
\]
It is easy to see that this chain is irreducible and aperiodic and that the stationary probability distribution $\Pi$ of this chain is given by
\[
\Pi(\tau_1,\dots,\tau_m) = \prod_{i=1}^m x_{\tau_i-1}.
\]
It is immediate that $\sum \Pi(\tau) = 1$.
Moreover, we obtain a $\Pi$-distributed word in at most $m$-steps.

Intuitively we think of $\tau_i$ as the 1 (from the left) that jumped to the front $i$ time steps ago. To formally define the lumping onto RJMC, define $S_i(w)$  for $w \in \B mb$ and $i \leq b$ to be the word obtained by the replacing the $i$'th 1 from the left in $w$ by 0 if such a 1 exists and $w$ otherwise. Then we claim that the map $\phi : \W mb \to \B mb$ defined recursively by
\[
\phi(\tau_1,\dots,\tau_m) = 
\begin{cases}
\emptyset & \text{if $m=0$}, \\
S_{\tau_1} (1,\,\phi(\tau_2,\dots,\tau_{m})) & \text{if $m \geq 1$},
\end{cases}
\]
defines the desired lumping. Using the intuitive understanding of the enriched chain this is not difficult to prove; the strategy is identical to that of \cite[Theorem~4.16]{ABCN-2015} with the words reversed. These facts prove Theorems~\ref{thm:RJMC-partfn} and \ref{thm:RJMC-conv}.

\section{Infinite Reverse Juggling Markov Chain} 
\label{S:single-infinite}
The Infinite Reverse Juggling Markov Chain (IRJMC) is the $m \to \infty$ limit of the RJMC in \refS{S:single-finite} and we will continue to use notation from there.
Let, as before, $b \in \mathbb{N}$ be the number of balls, and $\mathbb{P}(i) = x_i$ be an arbitrary probability distribution on $\{0,\dots,b\}$. Consider all semi-infinite binary words with $b$ ones and let 
$w = w_1,w_2,\dots$ be a word. Then the transitions of the IRJMC are as follows.  

\begin{enumerate}
\item With probability $x_{0}$, go to state $0,w_1,w_2,\dots$.

\item With probability $x_i$ for  $1 \leq i \leq b$, move the $i$'th one from the left to the front, replace it by a zero, and shift everything to the right.

\end{enumerate}

Let configurations be denoted by increasing $b$-tuples of integers , ${\bf n} = (n_1,\dots,n_b)$, indexing the positions of the ones. 
As in Section~\ref{S:single-finite}, let $z_i = x_0 + \cdots + x_i$ and $\bar z_i = 1- z_i =  x_{i+1} + \cdots + x_{b}$ for $0 \leq i \leq b$. Then, an equivalent description of the process is
\[
{\bf n} \to
\begin{cases}
(n_1+1,\dots,n_b+1) & \text{with probability $x_0$}, \\
(1, n_1+1,\dots,\widehat{n_i+1}, \dots, n_b+1) & \text{with probability $x_i$ for $1 \leq i \leq b$},
\end{cases}
\]
where $\hat x$ means that the element $x$ should be omitted.

\begin{prop}
\label{prop:single-pos-rec}
The IRJMC is positive recurrent if and only if $x_b > 0$.
\end{prop}

\begin{proof}
When $x_b > 0$, it suffices to show that there are exactly $b!$ ways to get from an arbitrary configuration $(n_1,\dots,n_b)$ to $(1,\dots,b)$ in $b$ steps, with total probability given by
\be \label{return-prob}
x_b (x_{b-1} + x_b) \cdots (x_1 + \cdots + x_b).
\ee
But this is clear, since at the first step, any of the $b$ balls can be moved to the first site with total probability $x_1 + \cdots + x_b$, following which any of the last $b-1$ balls can be moved to the first site with total probability $x_2 + \cdots + x_b$, and so on. 

Therefore, for any $m \geq b$, the probability of starting from and returning to $(1,\dots,b)$ in $m$ steps is given by \eqref{return-prob}. Therefore, if we let $M$ denote the transition matrix, we get that 
\[
\sum_{m=b}^\infty M^m_{(1,\dots,b),(1,\dots,b)},
\]
diverges, which implies that the chain is positive recurrent. On the other hand, if $x_b = 0$, then the last ball is at position at least $m$ after $m$ steps for all $m$. Therefore, all states are transient.
\end{proof}

\begin{theorem}
\label{thm:inf-RJMC-stat-dist}
Let ${\bf n} = (n_1,\dots,n_b)$ denote a configuration. Then
the stationary probability distribution $\pi$ of the IRJMC is given by 
\be \label{inf-single-stat-dist}
\pi({\bf n})= \frac{1}{Z_b}\prod_{i=0}^{b-1}  z_i^{n_{i+1}-n_{i}-1}.
\ee
where we set $n_0 = 0$ and $Z_b$ is the partition function.
\end{theorem}

\begin{proof}
Since the chain is positive recurrent by Proposition~\ref{prop:single-pos-rec} and aperiodic (there is a nonzero return probability to the state $(1,\dots,b)$), we have a unique stationary distribution. Therefore, it suffices to verify that $\pi(n_1,\dots,n_b)$ given by \eqref{inf-single-stat-dist} satisfies the master equation \eqref{mastereq}. The proof is very similar to that of Theorem~\ref{thm:RJMC-ss} in Section~\ref{S:single-finite}, and consequently, we will be sketchy.

If $n_1 > 1$, there is a single transition to ${\bf n}$ from 
$(n_1,\dots,n_b)$  with probability $x_0$, and it is easy to verify the master equation. If $n_1 = 1$, we combine the configurations which make a transition to ${\bf n}$ into $b$ different groups. For each $j \in [b-1]$ and each $\ell \in [n_j,n_{j+1}-2]$, we get a transition from $(n_2-1,\dots,n_j-1,\ell,n_{n+1}-1,\dots,n_b-1)$ with probability $x_j$. Adding these contributions to the master equation  for a fixed $j$, we obtain a total contribution of 
\[
\frac{1}{Z_b}\prod_{i=0}^{j-2} z_i^{n_{i+2}-n_{i+1}-1} \;
\left( z_j^{n_{j+1} - n_j - 1} - z_{j-1}^{n_{j+1} - n_j - 1} \right) 
\prod_{i=j+1}^{b-1} z_i^{n_{i+1}-n_{i}-1}.
\]
Here we have used the fact that $x_j=z_j-z_{j-1}$ and $(z_j-z_{j-1})\sum_{i=0}^r z_j^{r-i}z_{j-1}^i=z_j^{r+1} - z_{j-1}^{r+1}$.
We now add these terms for $1 \leq j \leq b-1$ using a telescoping argument to
obtain
\[
\frac{1}{Z_b}\prod_{i=1}^{b-1} z_i^{n_{i+1}-n_{i}-1} 
-
\frac{1}{Z_b}\prod_{i=0}^{b-2} z_i^{n_{i+2}-n_{i+1}-1}.
\]
The last group of configurations are the ones given by $(n_2-1,\dots,n_b-1,\ell)$ where $\ell \in [n_b,\infty)$, all of which make a transition with probability $x_b$. Adding this infinite contribution gives
\[
\frac{1}{Z_b}\prod_{i=0}^{b-2} z_i^{n_{i+2}-n_{i+1}-1},
\]
which, when added to the previous sum, returns $\pi({\bf n})$ and completes the proof.
\end{proof}

\begin{theorem}
\label{thm:inf-RJMC-part-fn}
The partition function of the chain is given by
\[
Z_b = \prod_{i=0}^{b-1} \frac{1}{ \bar z_i}.
\]
\end{theorem}

\begin{proof}
This is easily proved by induction on $b$. The case $b=1$ can be verified.
For fixed $(n_1,\dots,n_{b-1})$, the sum of $\pi(n_1,\dots,n_{b})$ over $n_b$ is a geometric progression, whose sum gives $1/ \bar z_{b-1}$.
\end{proof}

We obtain Knutson's result immediately as a special case of Theorems \ref{thm:inf-RJMC-stat-dist} and \ref{thm:inf-RJMC-part-fn}.

\begin{corollary}[{\cite[Theorem 1]{knutson-2018}}]
If the jump probabilities $x_i$ in the IRJMC are chosen to be
\begin{equation}
\label{special probs}
x_i =
\begin{cases}
q^{-b} & \text{if $i=0$}, \\
q^{-b+i}(1-q^{-1}) & \text{if $i>0$},
\end{cases}
\end{equation}
and we define $\ell(w)$ as the number of pairs $(i,j)$ where $i<j$ and $w_i = 0, w_j = 1$ , then
\[
\pi(w) = \frac{1}{q^{\ell(w)}} \prod_{i=1}^b \left( 1 - \frac{1}{q^i} \right).
\]
\end{corollary}

\begin{remark} Theorem \ref{thm:RJMC-ss} can also be deduced from Theorems \ref{thm:inf-RJMC-stat-dist} and \ref{thm:inf-RJMC-part-fn}. In fact, the RJMC of Section \ref{S:single-finite} is equivalent to studying the first $m$ positions of the IRJMC. 
\end{remark}

\begin{remark}
\label{rem:inf ball}
It is natural to ask if there is a variant of the IRJMC with infinite number of balls. We claim that such a chain will never be recurrent.
Assume, for contradiction, that a configuration with infinitely many balls recurs after $T$ jumps. Let $b$ be the largest label of a ball that jumped during these $T$ jumps. But then all balls numbered $b+1$ and higher will have moved to higher positions. Thus the new configuration cannot be the same as the one we started with. This argument also suggests that even if finitely many balls jump at each stage, the chain will not be recurrent. We would need infinite sets of balls to jump at some transitions.
\end{remark}

\section{Multispecies Reverse Juggling Markov Chain} 
\label{S:multi-finite} 
In this section, we study the finite Multispecies Reverse Juggling Markov Chain (MRJMC), for which we obtain a simple formula for the stationary distribution. The results obtained here will be useful in the study of the IMRJMC studied in Section~\ref{S:multi-infinite}. We note that the MRJMC is not a generalization of RJMC studied in Section~\ref{S:single-finite}. 
Assume we have $b$ balls with labels from a multiset $\mathcal M$ with $b_i$ elements $i$ for $1\le i\le T$, with $|\mathcal M|=b_1+\dots +b_T=b$. 
Let $S(\mathcal M)$ be the set of multipermutations of $\mathcal M$.
The MRJMC has as states multipermutations $\tau=(\tau_1,\dots,\tau_b) \in S(\mathcal{M})$ and is defined using two probability distributions; the {\em jump probabilities}  $s_1,\dots, s_b$ with $\sum_i s_i=1$ and $s_b>0$ and the {\em non-bump probabilities} $\alpha_1,\dots ,\alpha_{b-b_T}$.

Transitions in the MRJMC from $\tau$ are as follows.
With starting probability $s_j$ the ball in position $j$ starts a bumping path 
to the left. Assume there are $\ell$ and $r$ balls with smaller labels to the left and right of $\tau_{j}$ respectively.
Assume further that $1\le i_1<\dots<i_\ell<j<i_{\ell+1}\dots <i_{\ell+r}$ are numbers such that the balls with labels smaller than $\tau_j$ are positioned in ${i_1}<\dots <{i_{\ell+r}}$.

The ball $\tau_j$ now bumps the ball at position $i_k$, $1\le k\le \ell$ with probability $(1-\al_{r+\ell-k+1})\prod_{s=1}^{\ell-k}\al_{r+s}$ or is moved first in the permutation, position zero, without bumping any ball with probability $\prod_{s=1}^{\ell}\al_{r+s}$. Intuitively think of the ball moving left and bumping the first smaller ball with probability $1-\al_{r+1}$, if it does not it will bump the second one with probability $\al_{r+1}(1-\al_{r+2})$ etc. If a ball is bumped at position ${i_k}$, then repeat this step with $j=i_k$ to create a bumping path. The bumping path always ends with a ball placed at postition zero.
Then all balls in positions from zero to $j-1$ are moved one step right and we obtain a new permutation of $\mathcal M$.
In examples, we will often suppress the parentheses and write a state 
in one-line permutation form rather than vector form, i.e. $321321$ rather 
than $(3,2,1,3,2,1)$.

For $\tau\in S(\mathcal M)$ we define an inversion to be a pair $i<j$ such that $\tau_i>\tau_j$. 
Let $\inv(\tau)$ be the number of inversions of $\tau$.
Also let the
{\em code} (or Lehmer code) of $\tau$ be $\mathbf{c(\tau)} \equiv \mathbf c=(c_1,\dots,c_b)$, where $c_i:=\#\{k: i< k, \tau_i>\tau_k\}$; see \cite[page 30]{stanley-ec1}.
Let ${\bf \al}^{\bf c}(\tau)=\prod_{i=1}^b \al_{1}\al_{2}\dots \al_{c_i}$. 
One way to interpret the code of a multipermutation is that $c_i$ is the number of positions $j$, such that $i<j$ in $\tau$. 
Thus, if we specialise by setting all $\al_i=\al$ for all $i$, we just get ${\bf \al}^{\bf c}(\tau)=\al^{\inv(\tau)}$. 

\begin{example} 
${\bf \al}^{\bf c}(321321)=\al_1^4\al_2^3\al_3\al_4$. As an example of  a transition rate we give the following
$\P(1 2 1 3 2 1 3 1 \to 1 2 2 1 3 3 1 1)=x_7\al_2(1-\al_3)\al_3(1-\al_4)$, where the bumps happen in positions 
$7, 5$ and $1$. 
\end{example}

We illustrate the MRJMC with the following example.

\begin{example}
Consider the case of $T=3$ and ${\bf b} =(1,1,1)$. 
The transition matrix in the lexicographically ordered basis,
$\{123,132,213,231,312,321\}$, is given by
\begin{align*}
&s_3 \left(
\begin{array}{cccccc}
 \left(1-\alpha _1\right){}^2 & \alpha _1 \left(1-\alpha _2\right) & \alpha _1 \left(1-\alpha _1\right) & 0 & \alpha _1 \alpha _2 & 0 \\
 1-\alpha _1 & 0 & \alpha _1 & 0 & 0 & 0 \\
 1-\alpha _1 & 0 & 0 & \alpha _1 \left(1-\alpha _2\right) & 0 & \alpha _1 \alpha _2 \\
 1 & 0 & 0 & 0 & 0 & 0 \\
 0 & 1-\alpha _1 & 0 & \alpha _1 & 0 & 0 \\
 0 & 1 & 0 & 0 & 0 & 0 \\
\end{array}
\right) \\
+ & s_2 \left(
\begin{array}{cccccc}
 1-\alpha _1 & 0 & \alpha _1 & 0 & 0 & 0 \\
 0 & 1-\alpha _2 & 0 & 0 & \alpha _2 & 0 \\
 1 & 0 & 0 & 0 & 0 & 0 \\
 0 & 0 & 0 & 1-\alpha _2 & 0 & \alpha _2 \\
 0 & 1 & 0 & 0 & 0 & 0 \\
 0 & 0 & 0 & 1 & 0 & 0 \\
\end{array}
\right)
+ s_1
\left(
\begin{array}{cccccc}
 1 & 0 & 0 & 0 & 0 & 0 \\
 0 & 1 & 0 & 0 & 0 & 0 \\
 0 & 0 & 1 & 0 & 0 & 0 \\
 0 & 0 & 0 & 1 & 0 & 0 \\
 0 & 0 & 0 & 0 & 1 & 0 \\
 0 & 0 & 0 & 0 & 0 & 1 \\
\end{array}
\right),
\end{align*}
where, for the sake of readability, we have separately noted the transition matrix for each jump probability.
The stationary distribution is given by
\[
\frac{1}{\left(1+\alpha _1\right) \left(1+\alpha _1+\alpha _2 \alpha _1\right)}
\left(1, \alpha_1, \alpha_1, \alpha_1^2, \alpha_1 \alpha_2, \alpha_1^2 \alpha_2
\right).
\]
\end{example}

\begin{remark} \label{R:MSJMC}
Setting $s_b=1$ and $s_i=0$ for all other $i$ in the MRJMC gives a Markov chain on the same graph as the MSJMC  studied in \cite{ABCLN} but with different transitions.
\end{remark}

\begin{prop}
\label{prop:multi-finite-irred}
The MRJMC is irreducible and aperiodic if $s_b>0$ and $0 < \alpha_i < 1$ for all $i$.
\end{prop}

\begin{proof}
If $s_b = 1$ and all the $\alpha_i$'s belong to $(0,1)$, the underlying graph of the MRJMC is ergodic, since it is the same as that for the MSJMC by Remark \ref{R:MSJMC}. The MSJMC was shown to be irreducible and aperiodic in  \cite{ABCLN}. For more general jump probabilities, the graph has extra edges, and thus the chain continues to be irreducible and aperiodic.
\end{proof}

\begin{theorem} \label{T:multi-finite}
The stationary distribution $\pi$ of the MRJMC is given by
\[
\pi(\tau) = \frac{1}{Z_{\bf b}(\alpha_1,\dots,\alpha_{b-b_T})} {\bf \alpha^c}(\tau). 
\]
\end{theorem} 

\begin{remark}
The numerator of $\pi(\tau)$ refines the inversion number. In other words, if we set $\alpha_i = \alpha$ for each $i$, $\pi(\tau)$ will be proportional to $\alpha^{\inv(\tau)}$.
\end{remark}

We will prove Theorem~\ref{T:multi-finite} by verifying the master equation for a given state $\tau$ by using two refinements. First, let $T_t(\tau)$ be the set of all states such that there is a transition to $\tau$ with a bumping path starting in position $t$. Second, let 
$T_{r,t}(\tau)$ be the set of all states $\tau'$ such that there is a transition to $\tau$ with a bumping path starting in position 
$t$ and the last jump to position zero was from position $r$. 
We now claim that the following two lemmas hold.

\begin{lemma} 
\label{L:Refine-t} 
Summing over all incoming transitions to $\tau$ with a bumping path starting from position $t$ contribute $s_t{\bf \alpha^c}(\tau)$ to the master equation, that is
\[ \sum_{\tau' \in T_t(\tau)} {\bf \alpha^c}(\tau') \P (\tau' \to \tau)= s_t {\bf \alpha^c}(\tau).
\]
\end{lemma}

\begin{lemma} 
\label{L:Refine-rt} 
Summing over all incoming transitions to $\tau$ with a bumping path starting at position 
$t$ with the last jump from position $r$ gives 

\[ \sum_{\tau' \in T_{r,t}(\tau)} {\bf \alpha^c}(\tau') \P (\tau' \to \tau)=
\begin{cases}  
s_t {\bf \alpha^c}(\tau)\left(1-\al_{c_{r+1}(\tau)+1} \right) 
\ds \prod_{i< r,\; \tau_{i+1}>\tau_1} \al_{c_{i+1}(\tau)+1} & \text{if $r<t$},\\
\hfill  s_t {\bf \alpha^c}(\tau) \ds \prod_{i< t ,\; \tau_{i+1}>\tau_1} \al_{c_{i+1}(\tau)+1} & \text {if $r=t$}.
\end{cases}
\]
\end{lemma}

We will now prove Lemmas \ref{L:Refine-t} and \ref{L:Refine-rt} simultaneously.

\begin{proof}[Proof of Lemmas \ref{L:Refine-t} and \ref{L:Refine-rt}]
We will use induction on $b$. 
First, if $b=1$ the MRJMC has only one state and the lemmas are trivially true. The logic for the inductive step is the following. For a given $b$ we will prove that \refL{L:Refine-rt} implies \refL{L:Refine-t}. Then we will prove that  
\refL{L:Refine-rt} is implied by the inductive hyptothesis that \refL{L:Refine-t} is true for smaller values of $b$.
It might help to look at Example~\ref{eg:multi-finite} concurrently.

The first is relatively easy. 
If we have a transition $\tau'\to \tau$, with a bumping sequence starting in position $t$ in $\tau$ and with a last 
jump from position $r$, then, if $r<t$, the last jump must have been caused by an element larger than $\tau_1$, 
that is $\tau_{r+1}=\tau_r'>\tau_1$. 
Thus, the only possible values for $r<t$ are when $\tau_{r+1}>\tau_1$. 
Summing \refL{L:Refine-rt} over all those possible values of $r$ from $t$ and lower, it is easy to see that everything will cancel. To be more precise, for any $1 \leq x < t$, 
\[\sum_{r=x}^t \sum_{\tau' \in T_{r,t}(\tau)} {\bf \alpha^c}(\tau') \P (\tau' \to \tau)= s_t {\bf \alpha^c}(\tau)\prod_{i< x,\tau_{i+1}>\tau_1} \al_{c_{i+1}(\tau)+1}
\]
and for each term added on the left another of the factors on the right will disappear.

\medskip
The second part of the proof is a little more involved. Assume \refL{L:Refine-t} is true for all lengths smaller than $b$. 
Fix a 
state $\tau=(\tau_1,\ldots,\tau_b)$ and $2\le r<t\le b$ with $\tau_r>\tau_1$. Let also $\phi=(\tau_{r+1},\ldots,\tau_b)$. 
For a state $\tau'=(\tau_2,\ldots,\tau_r,\tau_1,\tau_{r+1}',\ldots, \tau_b')$ we let 
$\phi'=(\tau_{r+1}',\ldots, \tau_b')$. Then $\tau'\in T_{r,t}(\tau)$ if and only if $\phi'\in T_{t-r}(\phi)$. Let $s'_1,\dots, s'_{b-r}$ 
be the starting probabilities for the shorter bumping path in which $\phi'\to\phi$.
It should be clear that
\[
\P (\tau' \to \tau)=\frac{s_t}{s'_{t-r}}\P (\phi' \to \phi) (1-\al_{c_{r}(\tau)+1})\prod_{j=c_r(\tau')}^{c_1(\tau)-1}\al_{j+1}
\]
and
\[
{\bf\al^c}(\tau')={\bf\al^c}(\phi')\prod_{i=1}^r \al_1\dots\al_{c_i(\tau')}
\]
where
$c_1(\tau)=c_r(\tau')+ \#\{i:\tau_i<\tau_1, 2\le i\le r\}$ and 
\[
c_i(\tau')=
\begin{cases} 
c_{i+1}(\tau), & \text {if $\tau_{i+1}\le \tau_1$, $1\le i<r$ }\\
c_{i+1}(\tau)+1, & \text {if $\tau_{i+1}> \tau_1$, $1\le i<r$ }
\end{cases}.
\]
Thus we obtain
\begin{align*}
&\sum_{\tau' \in T_{r,t}(\tau)} {\bf \alpha^c}(\tau') \P (\tau' \to \tau)\\
=&\! \! \! \! \sum_{\phi' \in T_{t-r}(\phi)}\! \! \! \!  {\bf\al^c}(\phi')\prod_{i=1}^r \al_1\dots\al_{c_i(\tau')} \frac{s_t}{s'_{t-r}} \P (\phi' \to \phi) (1-\al_{c_{r}(\tau)+1})\prod_{j=c_r(\tau')}^{c_1(\tau)-1}\al_{j+1}\\
=&\prod_{i=1}^r \al_1\dots\al_{c_i(\tau)} \! \! \! \! \sum_{\phi' \in T_{t-r}(\phi)}\! \! \! \!  {\bf\al^c}(\phi') \frac{s_t}{s'_{t-r}}\P (\phi' \to \phi) (1-\al_{c_{r}(\tau)+1})\prod_{2\le i\le r, \tau_i>\tau_1}\al_{c_i(\tau)+1},
\end{align*}
which by induction from \refL{L:Refine-t} becomes
\begin{align*}
&\prod_{i=1}^r \al_1\dots\al_{c_i(\tau)} s_t {\bf\al^c}(\phi) (1-\al_{c_{r}(\tau)+1})\prod_{2\le i\le r, \tau_i>\tau_1}\al_{c_i(\tau)+1}\\ 
=&s_t {\bf\al^c}(\tau) (1-\al_{c_{r}(\tau)+1})\prod_{2\le i\le r, \tau_i>\tau_1}\al_{c_i(\tau)+1}.
\end{align*}

We also have to consider the case when $t=r\ge 1$. In this situation there is only one incoming transition to $\tau$
namely from $\tau'=(\tau_2,\dots,\tau_t,\tau_1,\tau_{t+1},\dots,\tau_b)$. Here we do not need induction; instead, we can directly compute the transition probability $\P (\tau' \to \tau)=s_t\prod_{j=c_r(\tau')}^{c_1(\tau)-1}\al_{j+1}$ and by studying the
change in inversions between $\tau$ and $\tau'$ we get the relation,
\[
{\bf\al^c}(\tau')={\bf\al^c}(\tau)\prod_{2\le i\le r, \tau_i>\tau_1}\al_{c_i(\tau)+1} \prod_{j=c_r(\tau')}^{c_1(\tau)-1}
\frac{1}{\al_{j+1}}. 
\]
Multiplying these two together we get exactly what is stated in \refL{L:Refine-rt}, namely
\begin{equation}
{\bf \alpha^c}(\tau') \P (\tau' \to \tau)=s_t {\bf\al^c}(\tau)\prod_{2\le i\le r, \tau_i>\tau_1}\al_{c_i(\tau)+1},
\end{equation}
which completes the proof.
\end{proof}

\begin{proof}[Proof of Theorem~\ref{T:multi-finite}]
Assuming Lemma~\ref{L:Refine-t} holds, consider the master equation \eqref{mastereq}. Summing over all $\tau'$ with a transition to $\tau$ means summing the lefthand side of the Lemma~\ref{L:Refine-t} over $t$ from 1 to $b$. The sum then becomes $\frac{1}{Z_b}\sum_{t=1}^b s_t {\bf \alpha^c}(\tau)=\frac{1}{Z_b}{\bf \alpha^c}(\tau)$, which is what we wanted to prove.
\end{proof}

\begin{example}
\label{eg:multi-finite}
Let $b=9, t=8$ and $\tau= 2 1 3 1 3 2 1 4 1$. We have that ${\bf \al^c}(\tau)=\al_1^5\al_2^4\al_3^3\al_4^2$ and we have the following incoming transitions.
\[\renewcommand{\arraystretch}{1.2}
\begin{array} {l l l l} 
 &  {\bf \al^c}(\tau') & \P (\tau' \to \tau) &\hspace{-1.0cm}\sum_{\tau' \in T_{r,t}(\tau)} {\bf \alpha^c}(\tau') \P (\tau' \to \tau)\\
 \hline
\hline
 r=8 & & &\\
 \tau'=
 \begin{tikzpicture}[line width=.5pt, scale=0.4, font=\small]
 \put(0,-4){
    \node at (0.5,-0.5) {$1$};
    \node at (1,-0.5) {$3$};
    \node at (1.5,-0.5) {$1$};
    \node at (2,-0.5) {$3$};
    \node at (2.5,-0.5) {$2$};
    \node at (3,-0.5) {$1$};
    \node at (3.5,-0.5) {$4$};
    \node at (4,-0.5) {$2$};
    \node at (4.5,-0.5) {$1$};
  \draw [->] (4,0) arc [radius=4, start angle=60, end angle= 120];};  
  \end{tikzpicture}
& \al_1^5\al_2^4\al_3^2\al_4^2\al_5 & s_8\al_2\al_3\al_4 & 
\hspace{-0.5cm} s_8\al_1^5\al_2^4\al_3^3\al_4^2\cdot\al_2\al_4\al_5 \\
 \hline
 r=7 &
& &\\
 \tau'= \begin{tikzpicture}[line width=.5pt, scale=0.4, font=\small]
 \put(0,-4){
    \node at (0.5,-0.5) {$1$};
    \node at (1,-0.5) {$3$};
    \node at (1.5,-0.5) {$1$};
    \node at (2,-0.5) {$3$};
    \node at (2.5,-0.5) {$2$};
    \node at (3,-0.5) {$1$};
    \node at (3.5,-0.5) {$2$};
    \node at (4,-0.5) {$4$};
    \node at (4.5,-0.5) {$1$};
  \draw [->] (3.5,0) arc [radius=3.5, start angle=60, end angle= 120];
  \draw [-] (4,0) arc [radius=0.25, start angle=10, end angle= 170];};  
  \end{tikzpicture} 
& \al_1^5\al_2^3\al_3^2\al_4^2\al_5 & s_8(1-\al_2)\al_2\al_3\al_4 & 
\hspace{-0.9cm} s_8\al_1^5\al_2^4\al_3^3\al_4^2\cdot(1-\al_2)\al_4\al_5 \\
 \hline
 r=4 & & &\\
 \tau'=
 \begin{tikzpicture}[line width=.5pt, scale=0.4, font=\small]
 \put(0,-4){
    \node at (0.5,-0.5) {$1$};
    \node at (1,-0.5) {$3$};
    \node at (1.5,-0.5) {$1$};
    \node at (2,-0.5) {$2$};
    \node at (2.5,-0.5) {$2$};
    \node at (3,-0.5) {$1$};
    \node at (3.5,-0.5) {$4$};
    \node at (4,-0.5) {$3$};
    \node at (4.5,-0.5) {$1$};
  \draw [->] (2,0) arc [radius=2, start angle=60, end angle= 120];
  \draw [-] (4,0) arc [radius=2, start angle=60, end angle= 120];};     
  \end{tikzpicture}
& \al_1^5\al_2^4\al_3\al_4\al_5 & s_8\al_2\al_3(1-\al_4)\al_3\al_4 &  \\
 \tau'=
  \begin{tikzpicture}[line width=.5pt, scale=0.4, font=\small]
 \put(0,-4){
    \node at (0.5,-0.5) {$1$};
    \node at (1,-0.5) {$3$};
    \node at (1.5,-0.5) {$1$};
    \node at (2,-0.5) {$2$};
    \node at (2.5,-0.5) {$2$};
    \node at (3,-0.5) {$1$};
    \node at (3.5,-0.5) {$3$};
    \node at (4,-0.5) {$4$};
    \node at (4.5,-0.5) {$1$};
  \draw [->] (2,0) arc [radius=2, start angle=60, end angle= 120];
  \draw [-] (3.5,0) arc [radius=1.5, start angle=60, end angle= 120];
    \draw [-] (4,0) arc [radius=0.25, start angle=10, end angle= 170];};     
  \end{tikzpicture}
 & \al_1^5\al_2^3\al_3\al_4\al_5 & s_8(1-\al_2)\al_2\al_3(1-\al_4)\al_3\al_4 &  \\
  &  & &\hspace{-0.5cm} s_8\al_1^5\al_2^4\al_3^3\al_4^2 \cdot (1-\al_4)\al_5   \\
 \hline
  r=2 & & &\\
 \tau'=
   \begin{tikzpicture}[line width=.5pt, scale=0.4, font=\small]
 \put(0,-4){
    \node at (0.5,-0.5) {$1$};
    \node at (1,-0.5) {$2$};
    \node at (1.5,-0.5) {$1$};
    \node at (2,-0.5) {$3$};
    \node at (2.5,-0.5) {$2$};
    \node at (3,-0.5) {$1$};
    \node at (3.5,-0.5) {$4$};
    \node at (4,-0.5) {$3$};
    \node at (4.5,-0.5) {$1$};
  \draw [->] (1,0) arc [radius=1, start angle=60, end angle= 120];
  \draw [-] (4,0) arc [radius=3, start angle=60, end angle= 120];
    };     
  \end{tikzpicture}
 & \al_1^5\al_2^4\al_3^2 & s_8\al_2\al_3\al_4(1-\al_5)\al_4 &  \\
 \tau'=
    \begin{tikzpicture}[line width=.5pt, scale=0.4, font=\small]
 \put(0,-4){
    \node at (0.5,-0.5) {$1$};
    \node at (1,-0.5) {$2$};
    \node at (1.5,-0.5) {$1$};
    \node at (2,-0.5) {$3$};
    \node at (2.5,-0.5) {$2$};
    \node at (3,-0.5) {$1$};
    \node at (3.5,-0.5) {$3$};
    \node at (4,-0.5) {$4$};
    \node at (4.5,-0.5) {$1$};
  \draw [->] (1,0) arc [radius=1, start angle=60, end angle= 120];
  \draw [-] (3.5,0) arc [radius=2.5, start angle=60, end angle= 120];
    \draw [-] (4,0) arc [radius=0.25, start angle=10, end angle= 170];
};  
\end{tikzpicture}
 & \al_1^5\al_2^3\al_3^2 & s_8(1-\al_2)\al_2\al_3\al_4(1-\al_5)\al_4 &  \\
  &  & &\!\!\!\!\! s_8\al_1^5\al_2^4\al_3^3\al_4^2 \cdot (1-\al_5)  \\
 \hline
\end{array}
\]
\end{example}

\begin{remark} \label{R:si}
It is surprising that, in Theorem \ref{T:multi-finite},  the stationary distribution of the MRJMC does not depend on the $s_i$'s. This means that at stationarity we may start the reverse juggling restricted to only the first $k$ positions for any $k<b$ with $s_k>0$ and we will stay at stationarity even though the balls in positions $k+1,\dots,b$ remain fixed. One might call this phenomenon ``partial mixing''.
We found that partial mixing also holds for the MSJMC \cite{ABCLN}, but it did not hold for some generalizations of the latter. It would be very interesting to understand better which Markov chains have this property of partial mixing. 
\end{remark}

The partition function of the MRJMC $Z_{\bf b} \equiv Z_{\bf b}(\alpha_1,\dots,\alpha_{b-b_T})$ can be computed recursively as follows.

\begin{theorem}
\label{thm:pf-multi-finite}
For $T=2$, the partition function is
\begin{equation} \label{eq:T=2}
Z_{p,q}=\sum_{0\le i_p\le \dots\le i_1\le q} \al_1^{i_1}\dots \al_p^{i_p}.
\end{equation}
For $T > 2$,
\begin{equation}
Z_{(b_1,\dots,b_T)} = \prod_{i=2}^{T} Z_{b_{1} + \cdots + b_{i-1},b_i}(\alpha_1,\dots,\alpha_{b_{1} + \cdots + b_{i-1}}) .
\end{equation}
\end{theorem}

\begin{proof} For $T=2$, assume we have a state $\tau$ with the 1's in positions $1\le j_p<j_{p-1}<\dots<j_1\le p+q$.  
In the product $\al^{\bf c}(\tau)=\prod_{i=1}^b \al_{1}\al_{2}\dots \al_{c_i}$ we will get one $\al_k$ for each 2 to the left 
of the $k$'th 1 from the right, thus $\al^{\bf c}(\tau)= \prod_{k=1}^p \al_k^{j_k-(p+1-k)}$. Setting $i_k=j_k-(p+1-k)$ 
we get an obvious bijection to the terms in the sum \eqref{eq:T=2}.

When $T>2$, define the map 
\[
\phi: S_{(b_1,\dots,b_T)} \rightarrow S_{(b_{1},b_2)} \times S_{(b_{1} + b_{2},b_3)} \times \cdots \times S_{(b_{1} + \cdots + b_{T-1},b_T)}
\]
as follows. Given a multipermutation $\tau$, we construct at $T-1$ tuple of multipermutations in letters $1$ and $2$, where the $i$'th entry consists of 
deleting letters greater than $i+1$, converting all $i$'s to $2$, and replacing all letters smaller than $i$ by 1. For example,
\[
\phi(3142414232) = (12122,2111121,1121212111).
\]
It is easy to see that $\phi$ is a bijection. Now, $\alpha^{c}(\tau)$ can be refined as
\[
\alpha^{c}(\tau) =\prod_{j=1}^p \alpha_j^{\#\{i| c_i(\tau) \geq j\}}= \prod_{j=1}^p \prod_{k=2}^T \alpha_j^{\#\{i| \tau_i = k,\; c_i(\tau) \geq j\}}.
\]
The idea is that the weight of $\tau$ can be obtained by computing  the weights of the simpler multipermutations. For the example above, 
\[
\alpha^c(3142414232) = \alpha_1^6 \alpha_2^4 \alpha_3^4 \alpha_4^3 \alpha_5^2
= \underbrace{\alpha^c(12122)}_{\alpha_1} \times
\underbrace{\alpha^c(2111121)}_{\alpha_1^2 \alpha_2 \alpha_3 \alpha_4 \alpha_5} 
\times 
\underbrace{\alpha^c(1121212111)}_{\alpha_1^3 \alpha_2^3 \alpha_3^3 \alpha_4^2 \alpha_5}.
\]
Now, the partition function can be written as
\[
Z_{(b_1,\dots,b_T)} = \sum_{\substack{ \tau^{(k)} \in S_{(b_{1} + \cdots + b_{k-1},b_k)} \\ \text{for $2 \leq k \leq T$}}} \prod_{k=2}^T 
\prod_{j=1}^p \alpha_j^{\#\{i| \tau^{(k)}_i = k, \; c_i(\tau^{(k)}) \geq j\}},
\]
from which it follows that the sums over the $\tau^{(k)}$'s can be performed separately, leading to the result.
\end{proof}

It is well-known (see, for example, \cite[Proposition 1.7.1]{stanley-ec1}) that the partition function, when specialised to all $\alpha_i = q$, becomes the $q$-multinomial coefficient. A direct application of Theorem \ref{thm:pf-multi-finite} gives the following product formula for the case of permutations. 

\begin{corollary} For permutations of length $T$, that is $\bf b=(1,\dots,1)$, we get
\[
Z_{(1,\dots,1)} = (1+\al_1)(1+\al_1+\al_1\al_2)(1+\al_1+\al_1\al_2+\al_1\al_2\al_3)\dots (1+\al_1+\dots +\al_1 \cdots \al_{T-1}).
\]
\end{corollary}

\section{Infinite Multispecies Reverse Juggling Markov Chain} 
\label{S:multi-infinite}
In this section, we consider an infinite variant of the MRJMC defined in Section~\ref{S:multi-finite}, which we call the Infinite Multispecies Reverse Juggling Markov Chain (IMRJMC). We will borrow most of the notation here from the MRJMC.

As before, we are given $b$ balls with labels from the set $\{1,\dots,T\}$ such that there are $b_i$ balls of type $i$. $\mathcal M$ is the multiset 
$\{1^{b_1},\dots,T^{b_T}\}$ with $|\mathcal M|=b$ 
and $S(\mathcal M)$ be the set of multipermutations of $\mathcal M$. 
A state is a pair $(\tau,\bf n)$, where $\tau=(\tau_1,\ldots,\tau_b) \in S(\mathcal M)$ and $\bf n \in \mathbb{N}^b$ a tuple of increasing positive integers $1\le n_1< \dots <n_b$. This should be interpreted as a configuration where, for each $j$, there is a ball labelled $\tau_j$ at position $n_j$. Since there is no upper bound on the $n_j$'s

We are given jump probabilities $x_0,x_1,\dots,x_{b}$, 
and non-bump probabilities $0 < \al_i < 1$ for $1\le i\le b-b_T$ 
The transition rules of the IMRJMC are very similar in spirit to those of MRJMC, and are described as follows.

\begin{enumerate}
\item With probability $x_{0}$, everything is moved one step right, i.e. $\tau$ is unchanged and $\mathbf{n} \to \mathbf{n+1} = (n_1+1,\dots,n_b+1)$. 

\item With probability $x_j$, the ball in position $n_{j}$ starts a bumping path to the left. Assume there are $\ell$ and $r$ balls with smaller labels to the left and right of $\tau_{j}$ respectively.
Assume further that $1\le i_1<\dots<i_\ell<j<i_{\ell+1}\dots <i_{\ell+r}$ are numbers such that the balls with labels smaller than $\tau_j$ are positioned in $n_{i_1}<\dots <n_{i_{\ell+r}}$.
The bumping rules and associated probabilities are identical to that of the MRJMC.
So in total ${\bf n}=(n_1,\dots,n_b)\to (1,n_1+1,\dots,n_{j-1}+1,n_{j+1}+1,\dots,n_b+1)$ and $\tau$ is changed as in the MRJMC in Section \ref{S:multi-finite}.
\end{enumerate}

As before, let ${\bf \al}^{\bf c}(\tau)=\prod_{i=1}^b \al_{1}\al_{2}\dots \al_{c_i}$ and 
$z_k=\sum_{i=0}^k x_i$. Also, let $n_0=0$.

\begin{prop}
\label{prop:multi-inf-irred}
Let $0 < \alpha_i < 1$ for all $i$ and $x_b>0$. 
Then the IMRJMC is irreducible, aperiodic and positive recurrent.
\end{prop}

\begin{proof}
Define the configuration
\be \label{spec-config}
c_0 = \big( (\underbrace{1,\dots,1}_{b_1},\dots,\underbrace{T,\dots,T}_{b_T}),
(1,\dots,b) \big).
\ee
Starting from any configuration $(\tau,{\bf n})$, one can reach $c_0$ in $b$ steps without any bumping by first moving balls labelled $T$ to the front, followed by those labelled $T-1$, and so on, until those labelled $1$. Similarly, starting from $c_0$, we can make a series of moves, again without bumping, in the order of $\tau$ starting from the right, and interspersing these with appropriate rightward moves depending on ${\bf n}$ so that one reaches $(\tau,{\bf n})$.
This proves irreducibility. Since there is a positive probability of going from the configuration $c_0$ to itself in one step, the chain is aperiodic.

The proof of positive recurrence is similar in spirit to that of Proposition~\ref{prop:single-pos-rec}. We will derive a lower bound for the probability of starting from $c_0$ in \eqref{spec-config} and returning to it in time $t$, when $t \geq b$. To do so, it suffices to bound the probability of starting from an arbitrary configuration $c = (\tau,{\bf n})$ and reaching $c_0$ in $b$ steps.

There are at least $b!$ ways to get from $c$ to $c_0$, corresponding to choices of the order of balls to be moved forward. For each such choice, we consider the bumping that sorts the prefix, i.e., when the $j$'th ball is moved forward,  bumping only happens at positions $1,\dots,j-1$ so that balls in positions $1,\dots,j$ are in increasing order. It is easy to see that this can be done in a unique way. We now give a lower bound for the probability of this move.

Let $\hat\alpha = \min_i (\alpha_i,1-\alpha_i)$. Then, for each choice above, the maximum bumping probability for these series of moves is no less than 
$\hat{\alpha}^{\sum_{i<j} b_i b_j}$.
By summing over all these choices of moves analogous to the proof of  Proposition~\ref{prop:single-pos-rec}, and including the bumping probabilities, we obtain that the probability of returning to $c_0$ in $t$ moves, for $t \geq b$ is bounded below by
\[
\hat{\alpha}^{\sum_{i<j} b_i b_j} \prod_{i=1}^b (x_i + \cdots + x_b).
\]
It follows that the sum over all $t \geq b$ diverges, and hence the chain is positive recurrent.
\end{proof}

We have the following.

\begin{theorem}
\label{thm:imrjmc-stat-dist}
The stationary probability distribution for the IMRJMC is given by
\[
\pi(\tau,{\bf n}) = \frac1Z {\bf \al}^{\bf c}(\tau)\prod_{k=1}^{b} z_{k-1}^{n_k-n_{k-1}-1}, 
\]
where $Z$ is the partition function.
\end{theorem}

\begin{proof} We argue that we can quickly reduce the verification of the master equation to the infinite single species 
chain and the finite multispecies chain and use Theorems \ref{thm:inf-RJMC-stat-dist} and \ref{T:multi-finite}.

Consider all bumping paths leading into a state $(\tau,{\bf n})$. All bumping paths starting from the same position
$j$ will have the same $x$ values and they can thus be taken outside the sum. The sum then runs over the same possible bumping paths as in the finite chain in the proof of Theorem \ref{T:multi-finite} and thus they sum to 
${\bf \al}^{\bf c}(\tau)$.

Summing over all possible ways to come to a state $(\tau,\bf n)$ with a bumping path starting with a ball (any label) in position $j$, $n_t-1<j<n_{t+1}-1$ is equal to
\[
\frac{x_t}Z {\bf \al}^{\bf c}(\tau) \sum_{j=n_t}^{n_{t+1}-2} p(n_2-1,\dots,n_t-1,j,n_{t+1}-1,\ldots,n_b-1), 
\]
where $p(n_1,\ldots,n_b)=\prod_{k=1}^{b} z_{k-1}^{n_k-n_{k-1}-1} $. 
This is exactly as in the proof of the stationary distribution of the RJMC in Theorem \ref{thm:inf-RJMC-stat-dist}, and this completes the proof.
\end{proof}

\begin{remark}
\label{rem:special-cases}
\begin{enumerate}
\item Setting $\al_i=0$ for all $i$ makes the IMRJMC reducible. The communicating class containing $\tau = (\underbrace{1,\dots,1}_{b_1},\dots,\underbrace{T,\dots,T}_{b_T})$ is irreducible, however. Since $\tau$ has no inversions, ${\bf \al}^{\bf c}(\tau) =1$. In this case, the chain on the tuples ${\bf n}$ is the IRJMC. 

\item Setting $x_0=0$ and $x_i=s_i$ for all other $i$ makes the IMRJMC reducible. The communicating class with ${\bf n} = (1,\dots,b)$ is irreducible, however. The chain on the multipermutations $\tau$, is exactly the same as the 
MRJMC. 
\end{enumerate}
\end{remark}

\begin{theorem}
\label{thm:pf-multi-inf}
The partition function of the IMRJMC is given by
\[
Z = Z_{(b_1,\dots,b_T)}(\alpha_1,\dots,\alpha_{b-b_T}) \prod_{i=0}^{b-1} \frac{1}{\bar{z}_i},
\]
where $Z_{(b_1,\dots,b_T)}(\alpha_1,\dots,\alpha_{b-b_T})$ is given by Theorem~\ref{thm:pf-multi-finite}.
\end{theorem}

\begin{proof}
Since the formula for the stationary distribution is a product of a function of the $x_i$'s and of the $\alpha_i$'s, we can perform the sums on these two separately. Both these sums have already been performed in Theorem~\ref{thm:inf-RJMC-part-fn} and \ref{thm:pf-multi-finite} respectively.
\end{proof}

We obtain Knutson's result as a corollary of Theorems~\ref{thm:imrjmc-stat-dist} and \ref{thm:pf-multi-inf}.
\begin{corollary}[{\cite[Theorem 2]{knutson-2018}}]
If the jump probabilities $x_i$'s are chosen as in \eqref{special probs} and the non-bump probabilities are chosen
to be $\alpha_i = 1/q$ for all $i$, then we obtain
\[
\pi(\tau,{\bf n})  = \frac{1}{q^{\ell(\tau,{\bf n})}} \left( \frac{1}{1-q^{-1}} \right)^b,
\]
where $\ell(\tau,{\bf n})$ is the number of pairs $(i,j)$ such that $i<j$ and $\tau_i > \tau_j$
plus the sum of $n_i - i+i$ over each $i$.
\end{corollary}	

In \cite{knutson-2018}, $\ell(\tau,{\bf n})$ was defined as the number of inversions of the juggling state 
where empty positions were considered to be labelled by infinity.   Theorem~\ref{thm:imrjmc-stat-dist} shows that this inversion number 
can be split into two parts; one coming with the $\alpha_i$'s from the multipermutation $\tau$, and the other with 
the $x_i$'s from the empty spaces.

\section{Remarks and Open Problems}
\label{S:open}

This work naturally leads to many open questions. We mention some of these here. 

\begin{open}
As we remarked at the beginning of Section~\ref{S:multi-finite}, the MRJMC is not a generalization of the RJMC. It is natural to ask for a multispecies generalization of the RJMC, but we have not yet been able to find one.
\end{open}

It is also natural to ask for a generalization of the IRJMC with infinitely many balls. Remark~\ref{rem:inf ball} suggests that a naive version of such a chain will not be irreducible.
A natural infinite generalization has been found for the forward juggling chain, known as the Unbounded Multivariate Juggling Markov Chain (UMJMC), in \cite{ABCN-2015}. 
\begin{open}
Is there a generalization of the IRJMC with infinitely many balls? 
\end{open}

If we change left and right and reverse the order of the labels, the transition graph of the MRJMC has the same edges as that of MSJMC in \cite{ABCLN}.  
Even though both model give nice product formulas for the stationary  distribution, the transition probabilities are of very different flavour.

\begin{open}
Find a common generalization of the MRJMC in this paper and the MSJMC in \cite{ABCLN}. 
\end{open}

The partition function of the MRJMC $Z_{b_1,\dots,b_T}$ is a generalisation of the inversion polynomial, which seems to be new. 
Even for permutations, we have not found this in the literature.

\begin{open}
What are the properties of the corresponding probability distribution and how does it relate to other well-known distributions on permutations, such as the Mallows distribution (where the probability of a permutation $\sigma$ is proportional to $q^{\inv{\sigma}}$).
\end{open}

Since the formula for the stationary distribution of the MRJMC is so simple, it is natural to ask if these could be extended by choosing more general probabilities.  It would be interesting to see how far the techniques in this work can be extended.

\begin{open}
For instance, if we let the $\al_i$'s depend, not only on the relative positions of balls being bumped, but also on the label of the balls themselves. Could that also lead to a solvable model? 
\end{open}

\section*{Acknowledgement}
We thank Allen Knutson for fruitful discussions. The first author (A.A.) acknowledges support from a UGC Centre for Advanced Study grant and DST grant DST/INT/SWD/ VR/P-01/2014. The second author (S.L.) has been supported by the Swedish Research Council, grant 621-2014-4780.

\bibliographystyle{alpha}
\bibliography{juggle}

\end{document}